\documentclass{amsart}
\usepackage{amsmath,amsthm,amssymb,IMjournal}
\usepackage{epsfig}  		
\usepackage{epic,eepic}       
\usepackage{tikz}
\usepackage{tabu}
\usepackage{subfig}
\usepackage{array}

\usepackage[
  separate-uncertainty = true,
  multi-part-units = repeat
]{siunitx}
\usetikzlibrary{arrows}
\usepackage{epsfig,}

\tikzset{
    vertex/.style = {
        circle,
        draw,
        outer sep = 3pt,
        inner sep = 3pt,
    },edge/.style = {->,> = latex'}
}

\def\rank{\mathop{\rm rank}}
\def\inertia{\mathop{\rm In}}

\newcommand{\rr}{\mathbb{R}}
\newcommand{\E}{\mathcal{E}}

\def\a{\alpha}
\def\tr{{\rm trace}}
\def\x{\widetilde{x}}
\def\M{\mathcal{M}}

\newtheorem{thm}{Theorem}[section]

\theoremstyle{definition}

\newtheorem{example}[thm]{Example}

\numberwithin{equation}{section}

\begin{document}

\title{Distance matrices perturbed by a Laplacian}

\author{\|Balaji |Ramamurthy|, Chennai,
        \|Ravindra |Bapat|, Delhi, \|Shivani |Goel|, Chennai.}

\rec { }

\dedicatory{Cordially dedicated to ...}

\abstract 
Let $T$ be a tree with $n$ vertices. To each edge of $T$, we assign a weight which is a positive definite matrix of some fixed order, say, $s$. Let $D_{ij}$ denote the sum of all the weights lying in the path connecting the vertices $i$ and $j$ of $T$. We now say that $D_{ij}$ is the distance between $i$ and $j$. Define $D:=[D_{ij}]$,
where $D_{ii}$ is the $s \times s$ null matrix and for $i \neq j$, $D_{ij}$ is the
distance between $i$ and $j$. Let $G$ be an arbitrary connected weighted graph with $n$ vertices, where each weight is a positive definite matrix of order $s$.
If $i$ and $j$ are adjacent, then define $L_{ij}:=-W_{ij}^{-1}$, where $W_{ij}$ is the weight of
the edge $(i,j)$. Define $L_{ii}:=\sum_{i \neq j,j=1}^{n}W_{ij}^{-1}$. The Laplacian of $G$ is now the 
$ns \times ns$ block matrix $L:=[L_{ij}]$. 
In this paper, we first note that $D^{-1}-L$ is always non-singular and then we   
prove that $D$ and its perturbation $(D^{-1}-L)^{-1}$ have many interesting properties in common.
\endabstract
 
\maketitle


\thanks
   The first author is supported by Department of science and Technology -India under the project MATRICS (MTR/2017/000342). 
\endthanks

\section{Introduction}\label{sec1}
Consider a finite, simple and undirected graph
$G=(V,\E)$, where $V$ is the set of vertices, and $\E$ is the set of edges. We write $V=\{1,\dotsc,n\}$ and 
$(i,j) \in \E$ if $i$ and $j$ are adjacent.
To an edge $(i,j)$ of $G$, we assign a weight $W_{ij}$ which is a positive definite matrix of order $s$. We now say that $G$ is a weighted graph. Define  
\begin{equation*}
V_{ij} = 
\begin{cases}
 W_{ij}^{-1}  & (i,j) \in \E \\
0_{s} & \mbox{else},
\end{cases}
\end{equation*}
where $0_s$ is the $s \times s$ null matrix. 
The Laplacian of $G$ is then the matrix

\[L(G):=
\left(
\begin{array}{cccccccc}
\sum_{k} V_{1k} & -V_{12} & -V_{13} & \cdots & -V_{1n} \\
-V_{21} & \sum_{k} V_{2k} & -V_{23} & \cdots & -V_{2n} \\
\cdots & \cdots & \cdots & \cdots & \cdots \\
-V_{n1} & -V_{n2} & -V_{n3} & \cdots & \sum_{k} V_{nk}
\end{array}
\right) .\]
Recall that a tree is a connected acyclic graph. Let $T$ be a tree with $n$-vertices.  The distance $S_{ij}$ between any two vertices $i$ and $j$ of $T$ 
is the sum of all
the weights lying in the path connecting $i$ and $j$. 
Define  
\begin{equation*}
D_{ij} = 
\begin{cases}
 S_{ij},  & i \neq j \\
0_{s} & i =j.
\end{cases}
\end{equation*}
Now the distance matrix of $T$ denoted by $D(T)$ is the $ns \times ns$ block matrix
with $(i,j)^{\rm th}$ block equal to $D_{ij}$. Distance matrices are well studied
when $s=1$, i.e., the weights are positive scalars. These matrices have a wide literature with numerous applications: see for example \cite{kirk}, \cite{fied} and references therein. Our objective in this paper is to go beyond the usual scalar case and
study much more general class of matrices for which the theory can be extended by some additional matrix theoretical techniques.

Several interesting properties of $D(T)$ and $L(T)$ are known. In addition, there are identities that connect $D(T)$ and $L(T)$. Distance matrices in this weighted set up
are introduced in \cite{bapat} and further investigated extensively in \cite{balajibapat}. Some of those important properties are listed below. These will be useful for proving our result. For brevity, we write $D=D(T)$ and $L=L(T)$.
\begin{enumerate}
\item[{\rm (P1)}]  Let $L^{\dag}$ be the Moore-Penrose inverse of $L$. Then, 
\[D_{ij}=L^{\dag}_{ii} + L^{\dag}_{jj} -2 L^{\dag}_{ij}. \] 
See Theorem 3.4 in \cite{balajibapat}.
\item[{\rm (P2)}] 
Let $\delta_i$ be the degree of the vertex $i$ and $\tau$ be the column vector with $i^{\rm th}$ component equal to $2-\delta_i$. Suppose $I_s$ is the 
$s \times s$ identity matrix. Then,
\[D^{-1}=-\frac{1}{2} L+ \frac{1}{2} \Delta R^{-1} \Delta^T, \]
where $R$ is the sum of all the weights in $T$ and $\Delta:=\tau \otimes I_s$: see Theorem 3.7 in \cite{balajibapat}.

\item[{\rm (P3)}] If $J$ is the block matrix with each block equal to $I_s$, then $D$ 
is negative definite on null-space of $J$. So,
$D$ has exactly $s$ positive eigenvalues: see section 2.3 in \cite{balajibapat}.

\item[{\rm (P4)}] If $G$ is connected, then $L(G)$ is positive semidefinite, $LJ=0$ and $\rank(L)=ns-s$. This can be proved easily. Now, it follows that column space of $L$ is contained in $\M$.
\end{enumerate}

\subsection{Objective of the paper}
Let $T$ be a weighted tree and $G$ be a weighted graph with $n$ vertices. Assume $G$ is connected. As before we shall write $D$ for $D(T)$ and $L$ for $L(G)$. The blocks of
$D$ and $L$ will be $D_{ij}$ and $L_{ij}$ which are $s \times s$ matrices.
We first show in this paper that $D^{-1}-L$ is always non-singular. Define $F:=(D^{-1}-L)^{-1}$. We say that $F$ is a perturbation of $D$. By performing certain numerical experiments, we observed that any perturbation $F$ has the following properties. 
\begin{enumerate}
\item[{\rm (a)}] Each block of $F$ is positive definite.
\item[{\rm (b)}] $F$ has exactly $s$ positive eigenvalues.
\item[{\rm (c)}] $F$ is negative definite on null-space of $J$.
\end{enumerate}
Items (a), (b) and (c) are satisfied by any distance matrix $D$. Our objective in this paper is to prove (a), (b) and (c) for any perturbation $F$ of $D$. When the weights are positive scalars, Bapat, Kirkland and Neuman  established a similar result in \cite{kirk}. It can be noted that the result in this paper
is a far reaching generalization of that result.

\subsection{Notation}
We fix the following notation.
\begin{enumerate}
\item[{\rm (N1)}] We say that $G$ is an $ns \times ns$ block matrix if $G$ can be partitioned
\[ 
\left[
\begin{array}{ccccc}
G_{11} & G_{12} & \cdots & G_{1n} \\
G_{12} & G_{22} & \cdots & G_{2n} \\
\vdots & \vdots & \vdots & \vdots  \\
G_{1n} & G_{2n} & \cdots & G_{2n}
\end{array}
\right]
\]
where each $G_{ij}$ is an $s \times s$ matrix. We now write $G=[G_{ij}]$. 

\item[{\rm (N2)}] Let $I_s$ denote the identity matrix of order $s$. Fix a positive integer $n$. Now, $J$ will be the 
$ns \times ns$ block matrix $[I_s]$. 

\item[{\rm (N3)}] We use $\M$ to denote the null space of $J$. 

\item[{\rm (N4)}] Define $U:=e\otimes I_s$, where $e$ is the column vector of all ones in $\rr^{n}$.
Thus, $J$ can be written $[U,\dotsc,U]$. Let $e_i$ be the $n$-vector with $1$ in the position $i$ and zeros elsewhere and $E_i:= e_i \otimes I_s$. 

\item[{\rm (N5)}] The transpose of a matrix $A$ is denoted by $A'$. 

\item[{\rm (N6)}] If $A$ is a symmetric matrix, we use $\inertia(A)$ to denote the inertia of $A$. We write
$\inertia(A)=(n_{-}(A),n_{z}(A),n_{+}(A))$, where $n_{-}(A)$, and $n_{+}(A)$ are the number of negative and positive eigenvalues of $A$, respectively and $n_{z}(A)$
is the nullity of $A$. 

\item[{\rm (N7)}] Suppose $m$ is a positive integer. The notation
$[m]$ will denote the finite set $\{1,\dotsc,m\}$.

\item[{\rm (N8)}] If $\Delta_1,\Delta_2 \subseteq [n]$, then $G[[\Delta_1,\Delta_2]]=[X_{ij}]$ will denote the $|\Delta_1| \times |\Delta_2|$ block matrix, where $(i,j) \in \Delta_1 \times \Delta_2$. If $\Delta=\Delta_1=\Delta_2$, then we simply write
$G[[\Delta]]$ for $G[[\Delta_1,\Delta_2]]$.
\end{enumerate}

\section{Result}
We now prove our main result.
\begin{thm}\label{T:prop:Dinv-L}
Let $T$ be a weighted tree with $n$-vertices, where each weight is of order $s$. Let
$D=[D_{ij}]$ be the distance matrix of $T$. Suppose $G$ is a connected and weighted graph
with $n$-vertices, where each weight is of order $s$. Let
 $L=[L_{ij}]$ be the Laplacian matrix of $G$. For any $\beta \geq 0$, the following are true.
\begin{enumerate}
    \item[\rm(i)] $D^{-1}-\beta L$ is non-singular.
    \item[\rm{(ii)}] $\inertia{(D^{-1}-\beta L)} = (ns-s,0,s)$.
    \item[\rm{(iii)}] 
Let $i \in [n]$ and $\Delta := [n] \smallsetminus \{i\}$. Define $F:=D^{-1}-\beta L$. Then, $F[[\Delta]]$    
    is negative definite.
    \item[\rm{(iv)}] The bordered matrix $G:= \left[
{\begin{array}{rrrrrrr}
(D^{-1}-\beta L)^{-1} & U \\
U' & 0 
\end{array}}
\right]$
is non-singular and has exactly $s$ positive eigenvalues.
    \item[\rm{(v)}] $(D^{-1}-\beta L)^{-1}$ is negative semidefinite on $\M$.
    \item[\rm{(vi)}] Every block in $(D^{-1}-L)^{-1}$ is positive definite.
\end{enumerate}
\end{thm}

\begin{proof}
\begin{enumerate}
    \item[\rm(i)] Let $x \in \rr^{ns}$ be such that
\[x'(D^{-1}-\beta L) = 0.\] 
Put $y=D^{-1}x$. Then, $y'=\beta x' L$ and so, by (P4) $y \in \M$.
By (P1), $D$ is negative semidefinite on $\M$, and hence $y'Dy \leq 0$. This implies that $x'D^{-1}x \leq 0$ and hence $x'Lx \leq 0$. As $L$ is  positive semidefinite, $x'Lx = 0$, and therefore,  $Lx=0$. This leads to $x'D^{-1}=0$ and hence $x=0$. Thus we get (i).

    \item[\rm(ii)] By (P2) and (P3), $\inertia{(D^{-1})} = (ns-s, 0, s)$. In view of (i),  $D^{-1}-\delta L$ is non-singular for any $\delta \geq 0$. Using the continuity of eigenvalues, we get $$\inertia{(D^{-1}-\beta L)} = (ns-s,0,s).$$
     
     \item[\rm(iii)] Without loss of generality, we assume $i=n$. Now
     $\Delta=\{1,\dotsc,n-1\}$. We complete the proof by showing that
      $L[[\Delta]]$ is positive definite and
     $D^{-1}[[\Delta]]$ is negative semidefinite. Let $L[[\Delta]]x = 0$ for some nonzero $x$ in $\rr^{ns-s}$.
      Define $\x \in \rr^{ns}$ by $\x:=(x,0)'$. Then, $\x^{'} L \x=0$. So, $L \x=0$. Hence by (P4), $\x$ is an element in column space of $J$. This implies that $\x$ is of the form $(p,p,\ldots,p)'$, for some $p \in \rr^s$.  
Since $\x=(x,0)'$, this means that $x=0$. Thus, $L[[\Delta]]$ is positive definite.

We now show that $D^{-1}[[\Delta]]$ is negative semidefinite.
Since $\inertia{(D^{-1})} = (ns-s, 0, s)$, by interlacing theorem $D^{-1}[[\Delta]]$ can have at most $s$ nonnegative eigenvalues. As $D_{nn} = 0$, by a Theorem of Fiedler and Markham in \cite{fied&markh}, nullity of $D^{-1}[[\Delta]]$ must be $s$. Therefore, $D^{-1}[[\Delta]]$ has no positive eigenvalues and thus, $D^{-1}[[ \Delta]]$ is negative semidefinite. Hence, $F[[ \Delta]]=D^{-1}[[\Delta]]-L[[\Delta]]$ is negative definite.

\item[\rm(iv)] Define $F:=(D^{-1}-\beta L)^{-1}$. Then, $G= \left[
{\begin{array}{rrrrrrr}
F & U \\
U' & 0 
\end{array}}
\right]$. Let the Schur complement of $F$ in $G$ be $G/F$. Since $LU=0$, we have
\begin{equation*} \label{scomp}
\begin{aligned}
G/F &=-U'F^{-1}U \\
&=-U'(D^{-1}-\beta L)U \\
&=-U'D^{-1}U. 
\end{aligned}
\end{equation*}
By Corollary $2.8$ in \cite{balajibapat},
$U'D^{-1}U$ is positive definite, and hence 
\[\inertia(G/F)=(s,0,0). \] By (ii), $\inertia(F)=(ns-s,0,s)$.
In view of Haynsworth inertia additivity formula,
\begin{equation*}
\inertia{(G / F)} = \inertia{(G)}-\inertia{(F)} ,
   \end{equation*}
and so,
\begin{equation*}
\begin{aligned}
\inertia(G) &=\inertia(G/F) + \inertia(F) \\
&= (s,0,0)+(ns-s,0,s) \\
&=(ns,0,s). \\
\end{aligned}
\end{equation*}
This completes the proof of (iv).

\item[\rm(v)] We now show that $(D^{-1}-\beta L)^{-1}$ is negative semidefinite on $\M$. If possible, let $p \in \rr^{ns}$ be such that 
\[p \in \M~~\mbox{and}~~p'(D^{-1}-\beta L)^{-1}p > 0.\] 
Now consider the following subspace of $\rr^{ns+s}$:
\[W:= \{(\alpha p,y)' : \alpha \in \rr ~\mbox{and}~ y \in \rr^s \}. \]
Let $v:=(\delta p,y)'$ be an arbitrary non-zero vector in $W$. 
As before, define \[G= \left[
{\begin{array}{rrrrrrr}
(D^{-1}-\beta L)^{-1} & U \\
U' & 0 
\end{array}}
\right].\]
Since $U'p=0$, we see that  
\begin{equation*}
\begin{aligned}
v'Gv &=\delta^2 p'(D^{-1}-\beta L)^{-1}p \\
& \geq 0.
\end{aligned}
\end{equation*}
So, $G$ is positive semidefinite on $W$. Since dimension of $W$ is $s+ 1$, $G$ will have at least $s+1$ non-negative eigenvalues. But from (iv), we see that $\inertia{(G)} = (ns,0,s)$. This is a contradiction. Hence, $(D^{-1}-\beta L)^{-1}$ is negative semidefinite on $\M$.

\item[\rm(vi)] 
Define $A:=(D^{-1}-L)^{-1}$ and let the $(i,j)^{\rm th}$ block of $A$ be written $A_{ij}$. We shall first prove that all the diagonal blocks of $A$
are positive definite. Let $H := A^{-1}$, and let $H_{ij}$ be the
$(i,j)^{\rm th}$ block of $H$.

We claim $A_{11}$ is positive definite. 
Define $\Delta:= \{2,\ldots,n\}$ and $Q:=H[[\Delta]]$. We note that 
\[H=\left[
\begin{array}{cccccc}
H_{11} & K \\
K' & Q \\
\end{array}
\right],
\]
and hence by a theorem of Fiedler and Markham \cite{fied&markh}, nullity of $Q$ and nullity of $A_{11}$ are equal. 
By (iii), $Q$ is negative definite and so, $Q$ is non-singular.
Hence, $A_{11}$ is non-singular.
Now it follows that 
\[\inertia(Q)=\inertia(A/A_{11}),\]
where $A/A_{11}$ is the Schur complement of $A_{11}$ in $A$.
In view of inertia additivity formula, (ii) and (iii), we see that 
\begin{equation*}
\begin{aligned}
\inertia(A_{11})&=\inertia(A) - \inertia(Q) \\
&=(ns-s,0,s) -(ns-s,0,0) \\
&=(0,0,s).
\end{aligned}
\end{equation*}
Thus, $A_{11}$ is positive definite. By a similar argument, we conclude that all diagonal blocks of $A$ are positive definite.

 For a non-zero vector $x \in \rr^s$, define $$G_x:= [x'A_{ij}x].$$ 
We claim that the off-diagonal entries of $G_x$ are non-zero. 
 Let $y = (y_1,y_2,\ldots,y_n)'$ be an element in $\{e\}^\perp$. For each $i \in [n]$, let $p^i:= y_ix$ and $p:= (p^1,p^2,\ldots,p^n)'$. 
As $$\sum_{i=1}^{n} p^{i}=\bigg(\sum_{i=1}^{n}y_{i}\bigg) x, $$ and $\sum_{i=1}^{n}y_{i}=0$, it follows that $Jp=0$ and hence
$p \in \M$. By (v), $A$ is negative definite on $\M$. Since
\(p'Ap = y'G_xy \), we now see that
$G_x$ is negative definite on $\{e\}^{\perp}$. 
This implies $G_x$ has at least $n-1$ negative eigenvalues. As each diagonal block $A_{ii}$ is positive definite, the diagonal entries of
$G_x$ are positive. So, $$\inertia(G_x)=(n-1,0,1).$$ By an application of interlacing theorem, we now see that all the off-diagonal entries
of $G_x$ are non-zero. This proves our claim. 

To this end, we have thus shown that if $A_{ij}$ is an off-diagonal block of $A$, then either $A_{ij}$ is positive definite or negative definite. 
We now claim that
$A_{ij}$ is positive definite.
Define $f:(0, \infty) \to \rr^{s \times s}$ by 
\[f(\alpha):= E_i' (D^{-1}-\alpha L)^{-1} E_j.\]
Note that $f(\alpha)$ is the $(i,j)^{\rm th}$ off-diagonal block of $(D^{-1}-\alpha L)^{-1}$. By a similar argument as above, we see that $f(\alpha)$ is either
positive definite or negative definite, for any $\alpha \in (0,\infty)$.
We now note that 
$$\tr(D_{ij})=\lim_{\alpha \downarrow 0}\tr(f(\a)).$$
As each block $D_{ij}$ is positive definite, $\tr(D_{ij})>0$. So, $\tr{(f(\delta))} > 0$ for some $\delta>0$. As $f(\a)$ is positive definite or negative definite for all $\a>0$, it follows that $\tr(f(\alpha))\neq 0$ for each $\alpha \in (0,\infty)$. Thus, $\tr(f(\alpha))>0$ for each $\alpha>0$ and hence $f(\alpha)$ is positive definite for all $\alpha>0$. Therefore, each block of $(D^{-1}-L)^{-1}$ is positive definite.
\end{enumerate}
\end{proof}

\subsection{Example}
To illustrate our result, we give the following example.  
\begin{example}\rm
Consider the following tree $T$ on four vertices. 
\begin{figure}[!h]
    \centering
    \begin{tikzpicture}[shorten >=1pt, auto, node distance=3cm, ultra thick,
   node_style/.style={circle,draw=black,fill=white !20!,font=\sffamily\Large\bfseries},
   edge_style/.style={draw=black, ultra thick}]
\node[vertex] (1) at  (-2,-1) {$1$};
\node[vertex] (2) at  (0,-1) {$2$};
\node[vertex] (3) at  (2,0) {$3$};
\node[vertex] (4) at  (2,-2) {$4$};
\draw  (1) edge node {$W_1$} (2);
\draw  (2) edge node {$W_2$} (3); 
\draw  (2) edge node {$W_3$} (4);
\end{tikzpicture}
\caption{Tree $T$.}\label{F:T1}
\end{figure}
Define $W_1 = 
\left[
{\begin{array}{rrrrrr}
8 & 6 \\
6 & 5
\end{array}}
\right],~ W_2 = 
\left[
{\begin{array}{rrrrrr}
1 & 1 \\
1 & 5
\end{array}}
\right],~ \text{and}~ W_3 = 
\left[
{\begin{array}{rrrrrr}
5 & 0 \\
0 & 5
\end{array}}
\right].$
Now, the distance matrix of $T$ is
    \begin{equation*}
       D= \left[
{\begin{array}{rr|rr|rr|rr}
0 & 0 & 8 & 6 & 9 & 7 & 13 & 6 \\
0 & 0 & 6 & 5 & 7 & 10 & 6 & 10 \\
\hline
 8 & 6 & 0 & 0 & 1 & 1 & 5 & 0 \\
6 & 5 & 0 & 0 & 1 & 5 & 0 & 5 \\
\hline
 9 & 7 & 1 & 1 & 0 & 0 & 6 & 1 \\
7 & 10 & 1 & 5 & 0 & 0 & 1 & 10 \\
\hline
 13 & 6 & 5 & 0 & 6 & 1 & 0 & 0 \\
6 & 10 & 0 & 5 & 1 & 10 & 0 & 0
\end{array}}
\right].
    \end{equation*}
Consider the graph $G$ on four vertices,
\begin{figure}[!h]
    \centering
    \begin{tikzpicture}[shorten >=1pt, auto, node distance=3cm, ultra thick,
   node_style/.style={circle,draw=black,fill=white !20!,font=\sffamily\Large\bfseries},
   edge_style/.style={draw=black, ultra thick}]
\node[vertex] (4) at  (-2,-1) {$4$};
\node[vertex] (3) at  (0,-1) {$3$};
\node[vertex] (1) at  (2,0) {$1$};
\node[vertex] (2) at  (2,-2) {$2$};
\draw  (1) edge node {$S_1$} (2);
\draw  (2) edge node {$S_2$} (3); 
\draw  (3) edge node {$S_4$} (4);
\draw  (3) edge node {$S_3$} (1);
\end{tikzpicture}
\caption{A connected graph $G$.}\label{F:G}
\end{figure}
where $S_1 = 
\left[
{\begin{array}{rrrrrr}
2 & 0 \\
0 & 2
\end{array}}
\right],~ S_2 = 
\left[
{\begin{array}{rrrrrr}
8 & 0 \\
0 & 8
\end{array}}
\right],~ S_3 = 
\left[
{\begin{array}{rrrrrr}
5 & -2 \\
-2 & 1
\end{array}}
\right],~ \text{and}~ S_4 = 
\left[
{\begin{array}{rrrrrr}
5 & -3 \\
-3 & 5
\end{array}}
\right].$ The Laplacian matrix $L(G)$ of $G$ is:
    \begin{equation*}
       L = \left[
{\begin{array}{rr|rr|rr|rr}
\frac{3}{2} & 2 & -\frac{1}{2} & 0 & -1 & -2 & 0 & 0 \\
2 & \frac{11}{2} & 0 & -\frac{1}{2} & -2 & -5 & 0 & 0 \\
\hline
 -\frac{1}{2} & 0 & \frac{5}{8} & 0 & -\frac{1}{8} & 0 & 0 & 0 \\
0 & -\frac{1}{2} & 0 & \frac{5}{8} & 0 & -\frac{1}{8} & 0 & 0 \\
\hline
 -1 & -2 & -\frac{1}{8} & 0 & \frac{23}{16} & \frac{35}{16} & -\frac{5}{16} & -\frac{3}{16} \\
-2 & -5 & 0 & -\frac{1}{8} & \frac{35}{16} & \frac{87}{16} & -\frac{3}{16} & -\frac{5}{16} \\
\hline
 0 & 0 & 0 & 0 & -\frac{5}{16} & -\frac{3}{16} & \frac{5}{16} & \frac{3}{16} \\
0 & 0 & 0 & 0 & -\frac{3}{16} & -\frac{5}{16} & \frac{3}{16} & \frac{5}{16}
\end{array}}
\right].
    \end{equation*}
Then, the matrix $(D^{-1}-L)^{-1}$ is: \\
\begin{equation*}
\left[
{\begin{array}{rr|rr|rr|rr}
\frac{3419893}{612184} & \frac{2467937}{612184} & \frac{3525525}{612184} & \frac{2430433}{612184} & \frac{3944573}{612184} & \frac{2285161}{612184} & \frac{4731635}{612184} & \frac{1962623}{612184} \\
\frac{2467937}{612184} & \frac{1957213}{306092} & \frac{2255293}{612184} & \frac{935945}{153046} & \frac{2218981}{612184} & \frac{1023631}{153046} & \frac{1853663}{612184} & \frac{2458795}{306092} \\
 & & & & & & & \\
\hline
 & & & & & & & \\
\frac{3525525}{612184} & \frac{2255293}{612184} & \frac{3037701}{612184} & \frac{1985813}{612184} & \frac{3651821}{612184} & \frac{2212445}{612184} & \frac{4430931}{612184} & \frac{1803363}{612184} \\
\frac{2430433}{612184} & \frac{935945}{153046} & \frac{1985813}{612184} & \frac{1566953}{306092} & \frac{2093885}{612184} & \frac{1966725}{306092} & \frac{1746783}{612184} & \frac{1159859}{153046} \\
 & & & & & & & \\
\hline
 & & & & & & & \\
\frac{3944573}{612184} & \frac{2218981}{612184} & \frac{3651821}{612184} & \frac{2093885}{612184} & \frac{3655573}{612184} & \frac{2328197}{612184} & \frac{4622251}{612184} & \frac{1831995}{612184} \\
\frac{2285161}{612184} & \frac{1023631}{153046} & \frac{2212445}{612184} & \frac{1966725}{306092} & \frac{2328197}{612184} & \frac{2026753}{306092} & \frac{1884375}{612184} & \frac{1242045}{153046} \\
 & & & & & & & \\
\hline
 & & & & & & & \\
\frac{4731635}{612184} & \frac{1853663}{612184} & \frac{4430931}{612184} & \frac{1746783}{612184} & \frac{4622251}{612184} & \frac{1884375}{612184} & \frac{3647621}{612184} & \frac{2294033}{612184} \\
\frac{1962623}{612184} & \frac{2458795}{306092} & \frac{1803363}{612184} & \frac{1159859}{153046} & \frac{1831995}{612184} & \frac{1242045}{153046} & \frac{2294033}{612184} & \frac{1968213}{306092} \\
 \\
\end{array}} 
\right].
\end{equation*}
We note that
each block in $(D^{-1}-L)^{-1}$ is positive definite, 
$\inertia{((D^{-1}-L)^{-1})} = (6,0,2)$, and 
$(D^{-1}-L)^{-1}$ is negative semidefinite on $\M$.
\end{example}

{\small
{\em Authors' address}:
{\em Balaji Ramamurthy}, Department of Mathematics, IIT Madras, Chennai, India.
 e-mail: \texttt{balaji5@\allowbreak iitm.ac.in} \\
 {\em Ravindra B. Bapat}, Theoretical Statistics and
Mathematics Unit, Indian Statistical Institute, Delhi, India.
 e-mail: \texttt{rbb@\allowbreak isid.ac.in}\\
{\em Shivani Goel}, Department of Mathematics, IIT Madras, Chennai, India.
 e-mail: \texttt{shivani.goel.maths@\allowbreak gmail.com}

}

\end{document}